\documentclass[10pt]{amsart}
\usepackage[margin=1.6in]{geometry}
\usepackage{amsthm}
\usepackage{mathtools}

\DeclarePairedDelimiter\floor{\lfloor}{\rfloor}

\usepackage[numbers]{natbib}

\usepackage{graphicx}
\usepackage{amssymb}
\usepackage{epstopdf}
\usepackage{hyperref}
\usepackage{xypic}

\newcommand\CC{\mathbb{C}}

\newcommand\FF{\mathbb{F}}

\newcommand\QQ{\mathbb{Q}}
\newcommand\RR{\mathbb{R}}

\newcommand\ZZ{\mathbb{Z}}

\newcommand\cF{\mathcal{F}}

\newcommand\cHH{\mathcal{H}}

\newcommand\cM{\mathcal{M}}

\newcommand\cT{\mathcal{T}}

\newcommand\cY{\mathcal{Y}}

\newcommand\Sp{\text{Sp}}

\newcommand\h{\mathfrak{h}}

\newtheorem{theorem}{Theorem}[section]
\newtheorem{lemma}[theorem]{Lemma}
\newtheorem{proposition}[theorem]{Proposition}
\newtheorem{corollary}[theorem]{Corollary}
\newtheorem{lemma*}{Lemma}
\theoremstyle{definition}

\newtheorem{remark}{Remark}

\newtheorem{notation*}{Notation}

\newtheorem{question}{Question}

\title{\texorpdfstring{The Infinite Topology of The Hyperelliptic Locus in Torelli Space }{The Infinite Topology of The Hyperelliptic Locus in Torelli Space}}
\author{Kevin Kordek}
\date{}
\begin{document}
\maketitle
\begin{abstract}
Genus $g$ Torelli space is the moduli space of genus $g$ curves of compact type equipped with a homology framing. The hyperelliptic locus is a closed analytic subvariety consisting of finitely many mutually isomorphic components. We use properties of the hyperelliptic Torelli group to show that when $g\geq 3$ these components do not have the homotopy type of a finite CW complex. Specifically, we show that the second rational homology of each component is infinite-dimensional. We give a more detailed description of the topological features of these components when $g=3$ using properties of genus 3 theta functions.
\end{abstract}
%
%
%
%
%
%
\section{Introduction}
Let $S_g$ denote a closed orientable reference surface of genus $g\geq 2$. Fix a hyperelliptic involution $\sigma$ of $S_g$. The \emph{hyperelliptic mapping class group} $\Delta_g$ is defined to be the centralizer of $\sigma$ in the mapping class group $\Gamma_g$ of $\sigma$. That is, $\Delta_g$ is defined to be the subgroup of all mapping classes that commute with $\sigma$. The \emph{hyperelliptic Torelli group} $T\Delta_g$ is the subgroup of $\Delta_g$ that acts trivially on $H_1(S_g,\ZZ)$. 

Much of the structure of $T\Delta_g$ remains unknown.  For example, it is unknown whether $T\Delta_g$ is finitely generated when $g\geq 3$. 
Since $T\Delta_2$ is equal to the genus 2 Torelli group, it follows from Mess's work \citep[Proposition 4]{mess1992torelli} that $T\Delta_2$ is a free group of infinite rank. Brendle-Childers-Margalit showed in \citep[Corollary 1.2]{brendle2013cohomology} that $T\Delta_3$ is not finitely presented, but is not known whether $T\Delta_g$ is finitely presented or if $H_2(T\Delta_g,\ZZ)$ is finitely generated when $g\geq 4$.

The Torelli space $\cT_g$ is the moduli space of smooth algebraic curves of genus $g$ equipped with a symplectic basis for the first integral homology. The period map $\cT_g\rightarrow \h_g$ to the rank $g$ Siegel upper half space is 2:1 and is branched along the locus of hyperelliptic curves. As we shall explain in Section \ref{moduliofhyperellipticcurves}, this locus consists of finitely many mutually isomorphic components, each of which is a model for the classifying space of $T\Delta_g$; we now fix one of these components, and denote it by $\cHH_g[0]$. Questions about $T\Delta_g$ are therefore equivalent to questions about the topology of $\cHH_g[0]$. 

The Torelli space $\cT_g^c$ of genus $g$ curves of compact type contains $\cT_g$ as an open submanifold. The closure  $\cHH_g^c[0]$ of $\cHH_g[0]$ in $\cT_g^c$ is smooth analytic variety that contains $\cHH_g[0]$ as a dense open subset. Each component of the hyperelliptic locus in $\cT_g^c$ is analytically isomorphic to $\cHH_g^c[0]$. Little is known about the topology of these components, aside from the result of Brendle-Margalit-Putman \citep[Theorem B]{brendle2014generators} that $\cHH_g^c[0]$ is simply-connected. The main result of this paper is the following theorem.
\begin{theorem}\label{theorem1}
For all $g\geq 3$, the rational homology group $H_2(\cHH^c_g[0],\QQ)$ is infinite-dimensional.
\end{theorem}
We comment that Theorem \ref{theorem1} stands in stark contrast to the genus 2 case, where it is known that $\cHH_2^c[0] = \cT_2^c$ is contractible \citep[p.11]{hain2006finiteness}.
\begin{corollary}
For all $g\geq 3$, $\cHH^c_g[0]$ does not have the homotopy type of a finite CW complex. 
\end{corollary}
Brendle-Childers-Margalit \citep[Main Theorem 2]{brendle2013cohomology} have shown that $H_{g-1}(T\Delta_g,\ZZ)$ is a free abelian group of infinite rank. This implies that for each $g\geq 2$, $\cHH_g[0]$ does not have the homotopy type of a finite complex and, in particular, that $H_2(T\Delta_3,\QQ)$ is an infinite-dimensional $\QQ$-vector space. 
It seems natural, in light of Theorem \ref{theorem1}, to ask if it is possible, through a study of the natural map $H_2(\cHH_g[0],\QQ)\rightarrow H_2(\cHH_g^c[0],\QQ)$, to extend this latter result to cases where $g \geq 4$.
\begin{question}\label{question1}
Are there values of $g\geq 4$ for which Theorem \ref{theorem1} implies that $H_2(\cHH_g[0],\QQ)$ is infinite-dimensional?\end{question}
An affirmative answer to Question \ref{question1} would imply that $T\Delta_g$ is not finitely presented for these values of $g$.
\\\\\indent 
When $g = 3$, we are able to give a more detailed description of the topological features of $\cHH_3^c[0]$. This is essentially due to the fact that $\cHH_3^c[0]$ has a geometric description as the zero locus of an even genus 3 thetanull in rank 3 Siegel space $\mathfrak{h}_3$. By studying the geometric properties of the zero loci of theta functions, we collect enough information to prove the following result.
\begin{theorem}\label{1.3}
Let $\cHH^c_3[0]$ denote a component of the hyperelliptic locus in $\cT_3^c$. Then $H_k(\cHH_3^c[0],\ZZ) = 0$ for all $k\geq 4$ and $H_3(\cHH_3[0],\ZZ)$ is a free abelian group.
\end{theorem}
%
%

We do not know the rank 2 \emph{integral} homology of $\cHH^c_3[0]$. If one could show that $H_2(\cHH^c_3[0],\ZZ)$ were free abelian, Theorem \ref{1.3} along with the homological form of Whitehead's Theorem \citep[Corollary 4.33]{hatcher2002algebraic} would imply that $\cHH^c_3[0]$ is homotopy equivalent to a bouquet of 2-spheres and 3-spheres. This would be the first complete determination of the homotopy type of $\cHH_g^c[0]$ for any value of $g\geq 3$. This result would also be interesting from the perspective of the theory of theta functions, since the topology of the zero loci of the even thetanulls of genus $g$ is poorly understood when $g\geq 3$. 
\begin{question}\label{question2}
Is $\cHH^c_3[0]$ homotopy equivalent to a bouquet of 2-spheres and 3-spheres? 
\end{question}
%
%
%
%

The paper is organized as follows. The requisite background material on mapping class groups, Teichm\"{u}ller theory, and the classical algebraic geometry of curves, abelian varieties and theta functions is presented in Section 2. In Section 3, we carry out homological computations needed for the proof of Theorem \ref{theorem1}, which is given at the end of Section 4. Sections 5 and 6 are dedicated to considerations in genus 3. In Section 5, we prove many results about the structure of the hyperelliptic locus in $\mathfrak{h}_3$. These structural results are used in Section 6 to compute integral homology of the components of the hyperelliptic locus. The proof of Theorem \ref{1.3} is given at the end of that section. 

\subsection{Acknowledgements}
Sections 5 and 6 of this paper were part of my thesis. I would like to thank my advisor Dick Hain for suggesting the problem of describing the topology of the hyperelliptic locus and for countless helpful discussions. I would also like the thank the referee for many helpful comments and corrections. 

\section{Background}\label{background}
\subsection{Hyperelliptic Torelli Groups}
Assume that $g\geq 2$. The \emph{mapping class group} $\Gamma_g$ is the group of isotopy classes of orientation-preserving diffeomorphisms of a fixed closed, orientable reference surface $S_g$ of genus $g$. Let $\sigma: S_g\rightarrow S_g$ denote a hyperelliptic involution of $S_g$. Define the hyperelliptic mapping class group $\Delta_g$ to be the centralizer in $\Gamma_g$ of $\sigma$. The Torelli group $T_g$ is the subgroup of $\Gamma_g$ that acts trivially on $H_1(S_g,\ZZ)$. It is a torsion-free group. The hyperelliptic Torelli group $T\Delta_g$ is defined to be the intersection $T_g\cap \Delta_g$. That is, $T\Delta_g$ is the subgroup of $\Delta_g$ that acts trivially on $H_1(T\Delta_g,\ZZ)$. Since $T\Delta_g$ is a subgroup of $T_g$, it is also torsion-free.

The natural representation of $\Gamma_g$ on $H_1(S_g,\ZZ)$ induces a map $\Gamma_g\rightarrow \Sp_g(\ZZ)$. This, in turn, induces a map $\Delta_g\rightarrow \Sp_g(\ZZ)$, and A'Campo \citep[Th\'{e}or\`{e}me 1]{acampo} proved that its image $G_g$ contains the level 2 subgroup $\Sp_g(\ZZ)[2]$. In addition, A'Campo proved that the image of $\Delta_g$ in $\Sp_g(\ZZ/2\ZZ)$ is isomorphic to $S_{2g+2}$, thought of as the symmetric group on the $2g+2$ Weierstrass points. 

For each integer $m\geq 0$, define the subgroup $\Delta_g[m]$ of $\Delta_g$ to be the kernel of the composition $\Delta_g\rightarrow \Sp_g(\ZZ)\rightarrow \Sp_g(\ZZ/m\ZZ)$. Then for each $m\geq 1$ there is an exact sequence
\begin{equation}\label{extension}
1\rightarrow T\Delta_g\rightarrow \Delta_g[2m]\rightarrow \Sp_g(\ZZ)[2m]\rightarrow 1.
\end{equation}
\subsection{Moduli of Hyperelliptic Curves}\label{moduliofhyperellipticcurves}
Let $\mathfrak{X}_g$ denote the Teichm\"{u}ller space of marked Riemann surfaces $f: C\rightarrow S_g$ of genus $g$. It is a contractible complex manifold of dimension $3g-3$ \citep[p.441]{arbarello2011geometry}. 
There is a natural left action of $\Gamma_g$ on $\mathfrak{X}_g$; it is holomorphic and properly discontinuous \citep[p.452]{arbarello2011geometry}. The moduli space of genus $g$ curves $M_g$ is, as an analytic variety, isomorphic to the quotient space $\Gamma_g\backslash \mathfrak{X}_g$ \citep[p.442]{arbarello2011geometry}. The hyperelliptic locus $\cHH_g$ in $M_g$ is the subvariety that parametrizes hyperelliptic curves. 

For a fixed hyperelliptic involution $\sigma$ of
 $S_g$, define $\mathfrak{X}_g^{\sigma}$ to be the fixed point locus of the action of $\sigma$ on $\mathfrak{X}_g$. 
It is smooth and contractible by \citep[Fact 4.1]{lochak}. Then the hyperelliptic locus in $\mathfrak{X}_g$ is the union
\begin{equation}
\mathfrak{X}_g^{hyp} = \coprod_{\sigma} \mathfrak{X}_g^{\sigma}
\end{equation}
where the union ranges over the set isotopy classes of hyperelliptic involutions of $S_g$ (of which there are countably many) \citep[p.14]{spivey}. The image of $\mathfrak{X}_g^{hyp}$ in $M_g$ is equal to $\cHH_g$,
which is therefore isomorphic to $\Delta_g\backslash \mathfrak{X}^{\sigma}_g$. 

\vspace{.1in}
%
Let $m$ be a non-negative integer. A level $m$ structure on genus $g$ curve $C$ is a choice of ordered basis for $H_1(C,\ZZ/m\ZZ)$ that is symplectic with respect to the intersection pairing. The moduli space of genus $g$ curves with a level $0$ structure is known as Torelli space; it is the quotient space $T_g\backslash \mathfrak{X}_g$. Because $T_g$ is a torsion-free group, $\cT_g$ is a complex manifold. We remark that the moduli space of genus $g$ curves with level $m$ structure is, as a variety, isomorphic to the quotient $\Sp_g(\ZZ)[m]\backslash \cT_g$.

\vspace{.1in}
The hyperelliptic locus $\cT_g^{hyp}$ in $\cT_g$ is the image of $\mathfrak{X}_g$ in $\cT_g$. This locus is the disjoint union of finitely many components, of which there are exactly $[\Sp_g(\ZZ):G_g]$, as is shown in 
\citep[p.13]{hain2006finiteness}. 
By A'Campo's result, this number is equal to 
\begin{equation*}
\frac{|\Sp_g(\ZZ/2\ZZ)|}{|S_{2g+2}|} = \frac{2^{g^2}\prod_{k=1}^g(2^{2k}-1)}{(2g+2)!}
\end{equation*}
(see, for example, \citep[p.13]{hain2006finiteness}).
Each component is isomorphic to the quotient $T\Delta_g\backslash \mathfrak{X}^{\sigma}_g$. This is the moduli space of hyperelliptic curves with level 0 structure. Because $\mathfrak{X}_g^{\sigma}$ is contractible and $T\Delta_g$ is a torsion-free group, $T\Delta_g\backslash \mathfrak{X}^{\sigma}_g$ is a $K(T\Delta_g,1)$. 

\vspace{.1in}
A curve of compact type is a stable nodal curve all of whose irreducible components are smooth and whose dual graph is a tree. By deformation theory, the $\cT_g$ can be enlarged to a complex manifold $\cT_g^c$ whose points parametrize genus $g$ curves of compact type with a level 
$0$ structure (a homology framing). The $\Sp_g(\ZZ)$-action on $\cT_g$ extends to $\cT_g^c$, and the moduli space $M_g^c$ of curves of compact type is the quotient $\Sp_g(\ZZ)\backslash \cT_g^c$ \citep[p.5]{hain2006finiteness}. The hyperelliptic loci $\cHH_g^c$ and $\cT_g^{c,hyp}$ in $M_g^c$ and $\cT_g^c$, respectively, are the closures of $\cHH_g$ and $\cT_g^{hyp}$ inside of $M_g^c$ and $\cT_g^c$, respectively.  It was conjectured by Hain \cite{hain2006finiteness} and later proved by Brendle-Margalit-Putman \cite{brendle2014generators} that the irreducible components of $\cT_g^{c,hyp}$ are all simply-connected.

In what follows, $\cHH_g^c[0]$ will be used to denote a fixed component of $\cT_g^{c,hyp}$ associated with a chosen hyperelliptic involution. Then $\cHH_g[0]:= \cHH_g^c[0]\cap \cT_g$ is component of $\cT_g^{hyp}$. The hyperelliptic loci $\cHH^c_g$ and $\cHH_g$ are obtained, respectively, as quotients $G_g\backslash \cHH_g^c[0]$ and $G_g\backslash \cHH_g[0]$.
\subsection{Abelian Varieties}
The material in this section is classical, and most of it can be found in, for example, \citep[Ch 2. Sec. 6]{griffiths2014principles}.
A \emph{polarized abelian variety} $A$ is an abelian variety along with a cohomology class $\theta\in H^2(A,\ZZ)$  (called a polarization) represented by a  positive, integral, translation-invariant $(1,1)$ form . A \emph{principally polarized} abelian variety (ppav) is a polarized abelian variety whose polarization, viewed as a skew-symmetric bilinear form on $H_1(A,\ZZ)$, can be put into the form
\begin{equation*}
\left(\begin{array}{cc}0 & I_{g\times g} \\ -I_{g\times g}& 0\end{array}\right)
\end{equation*}
with respect to some integral basis for $H_1(A,\ZZ)$. A \emph{framing} on a ppav $(A,\theta)$ is a choice of basis for $H_1(A,\ZZ)$ which is symplectic with respect to $\theta$. The symplectic basis will be denoted $\cF$.

Given a framed ppav $(A,\theta, \cF)$ one can find a unique basis of holomorphic 1-forms for which the period matrix is of the form 
\begin{equation*}
\left(\begin{array}{cc}\Omega \\ I_{g\times g}\end{array}\right).
\end{equation*}
In this case, $\Omega$ is called the \emph{normalized period matrix} of $(A,\theta,\cF)$. 
%
%
%
A framed ppav can be recovered, up to isomorphism, as the quotient $A_{\Omega_0}:=\CC^g/\Lambda(\Omega_0)$ where $\Lambda(\Omega_0) = \ZZ^g + \ZZ^g\Omega\subset \CC^g$ along with its canonical polarization \citep[pp.306-307]{griffiths2014principles} 

The space of normalized period matrices $\Omega$ of $g$-dimensional ppav's is known as \emph{Siegel space of rank $g$}. It will be denoted by $\h_g$. It is the open subspace of $M_{g\times g}(\CC)\cong \CC^{g(g+1)/2}$
consisting of symmetric matrices with positive-definite imaginary part \citep[Proposition 8.1.1]{birkenhake2004complex}. The symplectic group $\Sp_g(\ZZ)$ of integral symplectic $2g\times 2g$ matrices acts holomorphically on $\h_g$ by the following formula \citep[Proposition 8.2.2]{birkenhake2004complex}:
\begin{equation}
M\cdot \Omega = (A\Omega+B)(C\Omega+D)^{-1}\ \ \ \ \ \text{where} \ M = \left(\begin{array}{cc}A & B \\ C& D \end{array}\right)\in \Sp_g(\ZZ).
\end{equation}
The product of two ppav's $(A_1, \theta_1), (A_2, \theta_2)$ is the ppav whose underlying torus is $A_1\times A_2$ and whose polarization is the product polarization  $p_1^*\theta_1\oplus p_2^*\theta_2$, where $p_j : A_1\times A_2$ is projection onto the $j^{th}$ factor \citep[p.5]{hain2002rational}. A ppav which can be expressed as a product of two ppav's of positive dimension is called \emph{reducible}. The sublocus of $\h_g$ parametrizing reducible ppav's will be denoted $\h_g^{red}$.

\begin{proposition}$\left(\emph{Hain, \citep[p.4]{hain2006finiteness}}\right)$
There is an equality
$$\h_g^{red} = \bigcup_{j=1}^{[g/2]}\bigcup_{\phi\in \Sp_g(\ZZ)} \phi\left(\h_j\times \h_{g-j}\right).$$
\end{proposition}

\begin{remark}\label{remark11}

The following notation shall be used throughout the paper. If a framed (necessarily reducible) ppav has period matrix 
\begin{equation*}
\Omega = \left(\begin{array}{cc}\Omega_1 & 0 \\ 0& \Omega_2 \end{array}\right)
\end{equation*}
for matrices $\Omega_j \in \h_{g_j}$ with $g_1+g_2 = g$ and $g_j >0$, then we will write $\Omega = \Omega_1\oplus \Omega_2$.

\end{remark} 
\subsection{The Period Map}
There is a holomorphic map $\cT_g\rightarrow \h_g$ which sends the isomorphism class $[C;\cF]$ of a framed curve to 
the period matrix of its jacobian. It is called the \emph{period map}. The sharp form of the Torelli Theorem \citep[p.783]{mess1992torelli} states that the 
period map is a double branched cover of its image, and that it is injective along the locus of hyperelliptic curves. 

The period map extends to a proper holomorphic map $\cT_g^c\rightarrow \h_g$ (the \emph{extended period map}) that sends a framed curve of compact type to 
the period matrix of its (generalized) jacobian (cf. \citep[p.6]{hain2002rational}). The Torelli Theorem no longer holds in this setting; the fiber over a point in $\h_g^{red}$ will have positive dimension \citep[p.6]{hain2006finiteness}.

The boundary $\cT_g^{c,red}: = \cT_g^c - \cT_g$ is equal to the preimage in $\cT_g^c$ of $\h_g^{red}$ under the period map.

\subsection{Theta Functions}
A \emph{theta function of characteristic $\delta$} is a holomorphic function $\vartheta_{\delta}: \h_g\times \CC^g\rightarrow \CC$ defined by the following series:
\begin{equation}\label{thetaseries}
\vartheta_{\delta}(\Omega,z) = \sum_{m\in \ZZ^g}\exp(\pi i\left( (m+\delta')\Omega (m+\delta')^T + 2(m+\delta')(z+\delta'')^T\right))
\end{equation}
where $\delta = (\delta', \delta'') \in \frac{1}{2}\ZZ^{g}/\ZZ^g\oplus \frac{1}{2}\ZZ^{g}/\ZZ^g = \frac{1}{2}\ZZ^{2g}/\ZZ^{2g}\cong \FF^{2g}_{2}$ (cf. \citep[p.3]{grushevsky}). 

By differentiating the series (\ref{thetaseries}), one shows that $\vartheta_{\delta}$ satisfies the \emph{heat equation}:
\begin{equation}
2\pi i (1+\delta_{jk})\frac{\partial \vartheta_{\delta}}{\Omega_{jk}} = \frac{\partial ^2 \vartheta_{\delta}}{\partial z_j\partial z_k}
\end{equation}
where $\Omega = (\Omega_{jk})$ and $\delta_{jk}$ is the Kronecker delta, (cf. \citep[p.8]{grushevsky}).

The theta function $\vartheta_{\delta}$ is said to be \emph{even} (\emph{odd}) if for fixed $\Omega$ it is even (odd) as a function of $z$. Equivalently, $\vartheta_{\delta}$ is even (odd) if the number $\text{ord}_{z = 0}\vartheta_{\delta}(\Omega,z)$ is even (odd). If $\vartheta_{\delta}$ is even (odd), then the characteristic $\delta$ is also said to be even (odd). Equivalently, $\delta$ is even (odd) if and only if the number $\delta' (\delta'')^T\in \FF_2$ is equal to 0 (or 1) \citep[p.3]{grushevsky}. 
It can be directly checked that $\vartheta_{\delta}(\Omega,z)$ is a non-zero multiple of $\vartheta_0(\Omega, z+ \delta' + \delta''\Omega)$.

If a characteristic $\delta\in \frac{1}{2}\ZZ^{2g}$ can be written in the form
\begin{equation}
\delta = \left((\delta_1',\ldots, \delta_n'), (\delta_1'',\ldots,\delta_n'')\right)
\end{equation}
 for some characteristics $\delta_j = (\delta_j',\delta_j'')\in \frac{1}{2}\ZZ^{2g_j}$ with $g_1+\cdots + g_n = g$, then we write $\delta = \delta_1\oplus\cdots \oplus\delta_n$.
 
It follows directly from (\ref{thetaseries}) that if $\Omega$ is a block matrix of the form $\Omega_1\oplus\cdots \oplus \Omega_n$, with $\Omega_j\in\h_{g_j}$, then there are characteristics $\delta_j\in \frac{1}{2}\ZZ^{2g_j}$ such that $\delta = \delta_1\oplus\cdots \oplus \delta_n$ and 
\begin{equation}
\vartheta_{\delta}(\Omega,z) = \vartheta_{\delta_1\oplus\cdots \oplus\delta_n}(\Omega,z) = \prod_{1\leq j\leq n}\vartheta_{\delta_j}(\Omega_j, z(j))
\end{equation}
 where $z(j) = (z_{g_1+\cdots g_j+1},\ldots , z_{g_1+\cdots + g_{j+1}})$.
\\\\\indent 
The restriction of the theta function $\vartheta_{\delta}$ to $z = 0$ is the \emph{thetanull of characteristic $\delta$}. It is a holomorphic function
\begin{equation*}
\vartheta_{\delta}(-,0): \h_g\rightarrow \CC.
\end{equation*}
The thetanull $\vartheta_{\delta}(-,0)$ is called even or odd, respectively, if the corresponding theta function is. 
It follows from the definition that odd thetanulls vanish identically \citep[p.3]{grushevsky}. 

Theta functions enjoy special transformation properties with respect to the $\Sp_g(\ZZ)$-action on $\h_g\times \CC^g$ given by the formula
$M\cdot (\Omega, z) = \left(M\cdot \Omega,\  z\cdot (C\Omega+D)^{-1}\right)$. More specifically, we have the function $\vartheta_{\delta}(M\cdot \Omega,\ z\cdot (C\Omega+D)^{-1})$ is equal to $u\cdot \vartheta_{\delta'}(\Omega,z)$, where $u$ is a specific (very complicated) non-vanishing holomorphic function depending on $\delta, M, \Omega$ and $z$ and $\delta'$ is a characteristic depending on $\delta$ and $M$ given by another complicated formula. Precise formulas are given in \citep[p.227]{birkenhake2004complex}. Thetanulls obey similar transformation laws.

We remark here that, thanks to these transformation laws, given a fixed $\Omega_0\in \h_g$, the theta function $\vartheta_{\delta}(\Omega_0,z)$ defines a section of a holomorphic line bundle on the ppav $A_{\Omega_0}$ whose Chern class coincides with the polarization \citep[p.324]{birkenhake2004complex}.


Let $C$ be a smooth curve of genus $g$. Choose a theta characteristic $\alpha$ on $C$, 
i.e. the divisor class of a square root of the canonical bundle of $C$. The map $\text{Pic}^{g-1}(C)\rightarrow \text{Pic}^0(C)$ defined by $x\rightarrow x-\alpha$ sends $W_{g-1}$ to the locus $W_{g-1}-\alpha$ in $\text{Pic}^0(C)$. By Riemann-Roch, the divisor $W_{g-1}-\alpha$ is a symmetric theta divisor \citep[p.324]{birkenhake2004complex}. In other words, it is stable under the involution $x\rightarrow -x$ and its Chern class is equal to the polarization of $\text{Pic}^0(C)$ furnished by the cup product \citep[pp.327-328]{griffiths2014principles}.  This divisor will be denoted by $\Theta_{\alpha}$. This is a geometric formulation of \emph{Riemann's Theorem} \citep[Theorem 11.2.4]{birkenhake2004complex}.

The parity of $\Theta_{\alpha}$ matches that of the theta characteristic $\alpha$ \citep[Proposition 11.2.6]{birkenhake2004complex}. Let $\delta$ be a characteristic such that $\Theta_{\alpha}$ is the zero divisor of $\vartheta_{\delta}$. \emph{Riemann's Singularity Theorem} \citep[Theorem 11.2.5]{birkenhake2004complex} states that for $L\in \text{Pic}^{g-1}(C)$, 
\begin{equation*}
\text{mult}_{L}(W_{g-1}) = \text{mult}_{L-\alpha}(\Theta_{\alpha}) = \text{ord}_{L-\alpha}(\vartheta_{\delta}(\Omega_0,-)) = h^0(L).
\end{equation*}

\section{Homological Computations}
We will need the following two basic computations in order to carry out the calculations in this section. 
\begin{proposition}\label{kawazumi}\hspace{-.05in}$\left(\emph{\citep[p.1]{morifuji}}\right)$
For each $g\geq 2$ one has $H_j(\Delta_g,\QQ) \cong 0$ when $j = 1,2$.
\end{proposition}
\begin{proposition}\label{level2}
For each $g\geq 2$ there is an isomorphism $H_1(\Delta_g[2],\QQ)\cong \QQ^{g(2g+1)-1}$.
\end{proposition}
\begin{proof}
It follows from the Birman-Hilden Theorem \citep[p.269]{farb2011primer} that there is an extension 
$$1\rightarrow \ZZ/2\ZZ\rightarrow \Delta_g[2]\rightarrow \Gamma_{0,2g+2}\rightarrow 1$$
where $\Gamma_{0,2g+2}$ is the fundamental group of the moduli space $M_{0,2g+2}$ of smooth genus 0 curves with $2g+2$ marked points. 
This implies that  $H_1(\Delta_g[2],\QQ)\cong H_1(\Gamma_{0,2g+2},\QQ)$. The space $M_{0,2g+2}$ known to be a hyperplane complement with Poincar\'{e} polynomial 
$$(2t+1)\cdots (2gt+1)$$
(cf. \citep[p.2]{bergstrombrown}).
Since the coefficient of $t$ in the polynomial is equal to $\sum_{j=2}^{2g}j = \left(g(2g+1)-1\right)$, we are done. 
\end{proof}

We will need to compute the low degree homology of certain finite-index subgroups of $\Sp_g(\ZZ)$. This is accomplished by applying the following two theorems, both of which are due to Borel \citep[p.604]{haininfinitesimal}.
\begin{theorem}[Borel Stability]\label{borel}
For $k< g$, the rational cohomology $H^k(\Sp_g(\ZZ),\QQ)$ agrees with the degree $k$ part of the graded algebra $\QQ[x_2,x_6, x_{10},\ldots ]$, where the generator $x_{4k+2}$ has weight $4k+2$.
\end{theorem}
\begin{theorem}[Borel]\label{hainspg}
Suppose that V is an irreducible rational representation of the algebraic
group $\Sp_g$ and that $\Gamma$ is a finite index subgroup of $\Sp_g(\ZZ)$. If $k < g$, then
$H_k(\Gamma, V )$ vanishes when $V$ is non-trivial, and agrees with the stable cohomology of
$\Sp_g(\ZZ)$ when $V$ is the trivial representation.
\end{theorem}

Theorems \ref{borel} and \ref{hainspg} combined immediately give the following result.
\begin{proposition}\label{proposition1}
For any finite-index subgroup $\Gamma\subset \Sp_g(\ZZ)$ with $g\geq 3$ we have 
\begin{equation*}
H_1(\Gamma, \QQ) = 0 \ \ \ \ \text{and} \ \ \ \ H_2(\Gamma,\QQ) = \QQ.
\end{equation*}
\end{proposition}

With the aid of (\ref{extension}) and Propositions \ref{kawazumi} and \ref{level2} we will prove the following result.
\begin{proposition}\label{coinvariants}
For each $g\geq 3$ there are isomorphisms
\begin{equation*}
H_1(T\Delta_g,\QQ)_{\Sp_g(\ZZ)[2]} \cong \QQ^{g(2g+1)}\ \ \ \ \ \text{and}\ \ \ \ \ H_1(T\Delta_g,\QQ)_{G_g}\cong \QQ.
\end{equation*}
\end{proposition}
\begin{proof}
We will consider the 5-term exact sequences in homology associated to the extensions
\begin{align*}
1\rightarrow T\Delta_g\rightarrow &\Delta_g\rightarrow G_g\rightarrow 1\\
1\rightarrow T\Delta_g\rightarrow \Delta_g[2]&\rightarrow \Sp_g(\ZZ)[2]\rightarrow 1.
\end{align*}
 Since $H_2(\Delta_g,\QQ)$ vanishes by Proposition \ref{kawazumi}, it is ready checked that the image of the $H_2(\Delta_g[2],\QQ)\rightarrow H_2(\Sp_g(\ZZ)[2],\QQ)$ is trivial. It follows that we have an exact sequence
\begin{equation*}
0\rightarrow H_2(\Sp_g(\ZZ)[2],\QQ)\rightarrow H_1(T\Delta_g,\QQ)_{\Sp_g(\ZZ)[2]}\rightarrow H_1(\Delta_g[2],\QQ)\rightarrow 0.
\end{equation*}
The term on the left has dimension 1 and the term on the right has dimension $g(2g+1)-1$. Thus $H_1(T\Delta_g,\QQ)_{\Sp_g(\ZZ)[2]}$ has the claimed dimension. 

On the other hand, by Proposition \ref{kawazumi} we obtain an exact sequence
\begin{equation*}
0\rightarrow H_2(G_g,\QQ)\rightarrow H_1(T\Delta_g,\QQ)_{G_g}\rightarrow 0.
\end{equation*}
Since $H_2(G_g,\QQ)\cong \QQ$, we see that $H_1(T\Delta_g,\QQ)_{G_g}$ also has the claimed dimension.
\end{proof}

\section{The Combinatorics of The Boundary}
\begin{remark}
In this section, all homology groups are taken with $\QQ$-coefficients.
\end{remark}


The boundary $\cHH_g^c - \cHH_g$ of $\cHH_g^c$ is a divisor in $\cHH_g^c$ with $\floor*{\frac{g}{2}}$ boundary components. 
This follows quite readily from the discussion on p.390 of \cite{arbarello2011geometry}. 
These boundary components correspond to partitions of the form $g = g_1 + g_2$ with $g_1,g_2\geq 1$. Let $D_i$ denote the boundary component corresponding to the partition given by $g = i + (g-i)$, where $i = 1,\ldots, \floor*{\frac{g}{2}}$. Each $D_i$ is an irreducible variety. A generic point of $D_i$ corresponds to a curve obtained by identifying two smooth hyperelliptic curves of genus $i$ and $g-i$, respectively, at Weierstrass points \citep[p.390]{arbarello2011geometry}.

The preimage of $D_i$  in $\cHH_g^c[0]$ under the map $\cHH_g^c[0]\rightarrow \cHH_g^c$ is a countably infinite union of its irreducible components. Since $\cHH_g^c[0]$ is a smooth analytic variety, $H_2(\cHH_g^c[0], \cHH_g[0])$, is freely generated by the irreducible components of $\cHH_g^{c,red}[0]$. These components are permuted by the group $G_g$. Since $G_g$ acts on the pair $(\cHH_g^c[0], \cHH_g[0])$, we are able to prove the following.
\begin{proposition}\label{orbits}
For each $g\geq 2$ we have a $G_g$-module isomorphism
$$H_2(\cHH_g^c[0], \cHH_g[0])\cong \bigoplus_{j=1}^{\floor*{\frac{g}{2}}} M_j$$
where $M_j$ is a non-empty direct sum of induced representations of the form $\emph{Ind}^{G_g}_{H}\QQ$ for a subgroup $H$ of $G_g$. 
\end{proposition}
\begin{corollary}
For each $g\geq 2$ we have $\emph{dim}_{\QQ}H_2(\cHH_g^c[0],\cHH_g[0])_{G_g} \geq \floor*{\frac{g}{2}}$.
\end{corollary}
%
%
%
Since $\cHH_g^c[0]$ is simply connected, the long exact sequence of homology for the pair $(\cHH_g^c[0], \cHH_g[0])$ has a segment 
\begin{equation}
H_2(\cHH_g^c[0])\rightarrow H_2(\cHH_g^c[0], \cHH_g[0])\rightarrow H_1(\cHH_g[0])\rightarrow 0.
\end{equation}
\begin{proposition}
The map $H_2(\cHH_g^c[0])\rightarrow H_2(\cHH_g^c[0], \cHH_g[0])$ is non-zero. 
\end{proposition}
\begin{proof}
When $g\geq 4$, we take $G_g$-coinvariants in this sequence to obtain another exact sequence
\begin{equation*}
H_2(\cHH_g^c[0])_{G_g}\rightarrow H_2(\cHH_g^c[0], \cHH_g[0])_{G_g}\rightarrow H_1(\cHH_g[0])_{G_g}\rightarrow 0.
\end{equation*}
By Propositions and \ref{coinvariants} and \ref{orbits}, the image of $H_2(\cHH_g^c[0])_{G_g}$ in $H_2(\cHH_g^c[0], \cHH_g[0])_{G_g}$ must have dimension at least $\floor*{\frac{g}{2}}-1$. It follows that the image of $H_2(\cHH_g^c[0])$ in $H_2(\cHH_g^c[0], \cHH_g[0])_{G_g}$ is non-zero when $g\geq 4$. When $g = 3$ we consider instead the sequence
\begin{equation*}
H_2(\cHH_3^c[0])_{\Sp_3(\ZZ)[2]}\rightarrow H_2(\cHH_3^c[0], \cHH_3[0])_{\Sp_3(\ZZ)[2]}\rightarrow H_1(\cHH_3[0])_{\Sp_3(\ZZ)[2]}\rightarrow 0.
\end{equation*}
By Proposition \ref{coinvariants}, the rightmost term has dimension 21. 
A recent result of Hain \cite{hainpersonal} implies that the components of $\cHH_3^{c,red}[0]$ fall into exactly 56 types. These types are in bijection with the number of ways of partitioning the 8 Weierstrass points on a hyperelliptic genus 3 curve into subsets of 5 and 3. Two divisors of distinct type are not exchanged by the $\Sp_3(\ZZ)[2]$-action. It follows that $H_2(\cHH_3^c[0], \cHH_3[0])_{\Sp_3(\ZZ)[2]}$ has dimension at least 56 and that the map $H_2(\cHH_3^c[0])_{\Sp_3(\ZZ)[2]}\rightarrow H_2(\cHH_3^c[0], \cHH_3[0])_{\Sp_3(\ZZ)[2]}$ is non-zero.
\end{proof}
\begin{proposition}\label{proporbits}
For each $g\geq 2$, the vector space $H_2(\cHH_g^c[0],\cHH_g[0])$ has the property that the $G_g$-orbit of any irreducible component of $\cHH_g^{c,red}[0]$ spans an infinite-dimensional subspace.
\end{proposition}
\begin{proof}
By the uniqueness of the decomposition of an analytic variety into its irreducible components, every irreducible component $x$ of $\cHH_g^{c,red}[0]$ maps via the period map into a unique irreducible component $\xi$ of $\h_g^{red}$. Since the $\Sp_g(\ZZ)$-orbit of $\xi$ in the set of components of $\h_g^{red}$ is infinite and because $G_g$ is a finite-index subgroup of $\Sp_g(\ZZ)$, it follows that the $G_g$-orbit of $\xi$ is also infinite. Since the period map is $G_g$-equivariant, it follows that the $G_g$-orbit of $x$ in $H_2(\cHH_g^c[0], \cHH_g[0])$ must be infinite-dimensional. 
\end{proof}
\begin{proof}[Proof of Theorem \ref{theorem1}]
It follows from Proposition \ref{proporbits} that the $G_g$-orbit of any non-zero vector $v\in H_2(\cHH_g^c[0],\cHH_g[0])$ spans a subspace of infinite-dimension. Now assume that $v$ lies in the image of the map $H_2(\cHH_g^c[0])\rightarrow H_2(\cHH_g^c[0],\cHH_g[0])$. Since this map is $G_g$-equivariant, it follows that there is a subspace of $H_2(\cHH_g^c[0])$ that maps onto an infinite-dimensional subspace of $H_2(\cHH_g^c[0],\cHH_g[0])$. Thus $H_2(\cHH_g^c[0])$ has infinite dimension.
\end{proof}
\section{The Genus 3 Picture}

The hyperelliptic locus $\h_3^{hyp}$ in $\h_3$ is the image of $\cT_3^{c,hyp}$ under the period map $\cT_3^c\rightarrow \h_3$. The period map restricts to an injective immersion along each component of $\cT_3^{c,hyp}$. Later in this section, we will argue that the image of each component is the zero locus of exactly one of the 36 even genus 3 thetanulls. Each of these is a smooth analytic variety. Thus the period map induces an isomorphism of $\cHH_3^c[0]$ with the zero locus of an even thetanull. 
Our strategy is to exploit this relatively concrete geometric description of $\cHH_3^c[0]$ in order to study its topology, which we do in Section 6.

In this section we study the geometry of the hyperelliptic locus $\h_3^{hyp}$ in $\h_3$. We will in particular focus on the geometry of the irreducible components of $\h_3^{hyp}$ in order to extract information about their topology, and therefore the topology of $\cHH_3^c[0]$. Many of these results are due to Hain and can be found in the unpublished manuscript \cite{hain2015torelli}. 
However, to the best of my knowledge, the proofs of Corollary 5.4 and Lemma 5.6 presented here are original. 

\subsection{\texorpdfstring{The Geometry of $\h_3$.}{The Geometry of h3}}
Let $\h_3^{red}$ denote the locus of period matrices of reducible framed principally polarized abelian 3-folds. It is a closed subvariety of $\h_3$ of codimension 2. The action of $\Sp_3(\ZZ)$ on $\h_3^{red}$ permutes the irreducible components of $\h_3^{red}$ transitively. Each component is a translate of the component $\h_2\times \h_1\subset \h_3$, which is embedded via
\begin{equation*}
(\Omega, \tau)\rightarrow  \left(\begin{array}{cc}\Omega & 0 \\0 & \tau\end{array}\right) = \Omega\oplus \tau.
\end{equation*}
The stabilizer of $\h_2\times \h_1$ in $\Sp_3(\ZZ)$ is equal to $\Sp_2(\ZZ)\times \Sp_1(\ZZ)$. 
\\\\\indent 
The singular locus $\h_3^{red,sing}$ of $\h_3^{red}$ is precisely the locus parametrizing reducible framed ppav's which are a product of three genus 1 curves. It is a closed subvariety of codimension 3. Again the $\Sp_3(\ZZ)$-action permutes the components of $\h_3^{red,sing}$ transitively. Each component is a translate of the component $\h_1^3 = \h_1\times \h_1\times \h_1\subset \h_3$ which is embedded via
\begin{equation*}
(\tau_1, \tau_2, \tau_3)\rightarrow  \left(\begin{array}{ccc}\tau_1 & 0 &0 \\0 & \tau_2&0 \\ 0 & 0& \tau_3\end{array}\right). 
\end{equation*}
The stabilizer of $\h_1\times \h_1\times \h_1$ in $\Sp_3(\ZZ)$ equals $\Sp_1(\ZZ)^3 \rtimes S_3$, where $S_3$ denotes the symmetric group on three letters and $\Sp_1(\ZZ)^3 = \Sp_1(\ZZ)\times \Sp_1(\ZZ)\times \Sp_1(\ZZ)$.
\\\\\indent 
Hain has shown \citep[Lemma 11]{hain2002rational} that if $p\in \h_3$ lies on a component of $\h_3^{red,sing}$, then exactly three components of $\h_3^{red}$ intersect at $p$.  In fact, in a neighborhood of $p$ there are holomorphic coordinates 
\begin{equation*}
\left(\begin{array}{ccc}\tau_1 & z_1 & z_3 \\\ z_1 & \tau_2 & z_3 \\z_2 & z_3 & \tau_3\end{array}\right)
\end{equation*}
for $\h_3$ such that locally around $p$ these components are cut out by the equations
\begin{equation*}
z_1 = z_2 = 0 \ \ \ \ \ \ z_1 = z_3 = 0 \ \ \ \ \ \ z_2 = z_3 = 0.
\end{equation*}

\subsection{The Hyperelliptic Locus and the Period Map}
%
%
%
We recall some facts about the period map from Section \ref{background}. The period map $\cT_3\rightarrow \h_3$ is generically 2-1 onto its image and is branched over the locus of jacobians of smooth hyperelliptic curves. Its image is exactly $\h_3-\h_3^{red}$ \cite{hain2015torelli}. The period map extends to a proper, surjective, holomorphic map %
$\cT_3^c\rightarrow \h_3$, which will also be referred to as the period map, as the meaning will be clear from the context. Note that, though the fibers of $P$ over $\h_3-\h_3^{red}$ are finite, its fibers over $\h_3^{red}$ are 1-dimensional.
\\\\\indent 
Let $\cT_3^{c,hyp}$ denote the hyperelliptic locus in $\cT_3^c$. Standard arguments from deformation theory (c.f. \cite{arbarello2011geometry}, p.211) imply that $\cT_3^{c,hyp}$ is a smooth subvariety of $\cT_3^c$ of dimension 5. We record this as
\begin{lemma}\label{lemma45}
The hyperelliptic locus $\cT_3^{c,hyp}$ is a smooth divisor in $\cT_3^c$.
\end{lemma} 
Let $\h_3^{hyp}$ denote the image of $\cT_3^{c,hyp}$ in $\h_3$. It is the closure in $\h_3$ of the locus of period matrices of smooth hyperelliptic curves. 
\begin{lemma}[Hain, \cite{hain2015torelli}]
The hyperelliptic locus $\h_3^{hyp}\subset \h_3$ is the union of the zero loci of the 36 even thetanulls.
\end{lemma}
\begin{proof}
The subset $\h_3^{hyp}\cap (\h_3-\h_3^{red})\subset \h_3$ is the union of the zero loci of the even thetanulls $\vartheta_{\delta}(-,0)$ restricted to $\h_3-\h_3^{red}$. This is a consequence of the fact that a smooth curve of genus 3 is hyperelliptic if and only if it possesses a vanishing even theta characteristic, which is necessarily unique, 
cf. \citep[Proposition 1.4.6]{nart} and \citep[Proposition 5.3]{hartshorne1977algebraic}. On the other hand, every point of $\h_3^{red}$ is the period matrix of a framed hyperelliptic curve, and
every element of $\h_3^{red}$ is contained in the zero locus of some even thetanull. 
\end{proof}

It should be noted that $\Sp_3(\ZZ)$ acts transitively on the set of zero loci of the even thetanulls and therefore also on the components of $\h_3^{hyp}$. Thus the filtration
\begin{equation*}
\h_3\supset \h_3^{hyp}\supset \h_3^{red}\supset \h_3^{red,sing}
\end{equation*}
is stable under the action of $\Sp_3(\ZZ)$.
\\\\
\indent Henceforth, we shall denote by $D_{\alpha}$ the sublocus of $\h_3^{hyp}$ which is the zero locus of the even thetanull $\vartheta_{\delta_{\alpha}}(-,0)$. The next result will imply that the $D_{\alpha}$ are precisely the irreducible components of $\h_3^{hyp}$. 
\begin{proposition}[Hain, \cite{hain2015torelli}]
The zero locus $D_{\alpha}$ of $\vartheta_{\delta_{\alpha}}(-,0)$ is non-singular.
\end{proposition}
\begin{proof}
This is a straightforward application of the heat equation for theta functions and the Riemann Singularity Theorem.  Suppose first that $\Omega\in D_{\alpha}\cap (\h_3-\h_3^{red})$. Then $\Omega$ is the period matrix of a non-singular framed hyperelliptic curve $C$ and the theta divisor $\Theta_{\alpha}$ satisfies
\begin{equation*}
\text{mult}_0\Theta_{\alpha} = h^0(\alpha) = 2.
\end{equation*}
By the heat equation, at least one of the numbers
\begin{equation*}
 \frac{\partial^2 \vartheta_{\delta_{\alpha}}}{\partial z_j\partial z_k}(\Omega,0)
 = 2\pi i(1+\delta_{jk}) \frac{\partial \vartheta_{\delta_{\alpha}}}{\partial \Omega_{jk}}(\Omega,0)
\end{equation*}
is non-zero. Thus $D_{\alpha}$ is non-singular at $\Omega$. 

$\hspace{.1in}$ If $\Omega\in D_{\alpha}\cap \h_3^{red}$, the situation is somewhat more complicated. By the transformation law for thetanulls, we may assume that $\Omega = \Omega_0\oplus \tau_0\in \h_2\times \h_1$ and that $\vartheta_{\delta_{\alpha}}(\Omega,0) = 0$. We will argue that $\vartheta_{\delta_{\alpha}}(-,0)$ vanishes to first order at $\Omega$. There are two cases to consider. The first is that $\Omega_0\in \h_2-\h_2^{red}$ and the second is that $\Omega_0\in \h_2^{red}$. Assume that $\Omega_0\in \h_2-\h_2^{red}$.
Since $\Omega$ is in block form, we have 
$$\vartheta_{\delta_{\alpha}}(\Omega,z) = \vartheta_{\delta_{\alpha_1}}(\Omega_0,z')\vartheta_{\delta_{\alpha_2}}(\tau_0,z'')$$
where $z' = (z_1,z_2)$ and $z'' = z_3$.
Since $\delta_{\alpha}$ is even and vanishing, it must be the case that $\delta_{\alpha_1}$ and $\delta_{\alpha_2}$ are either both even or both odd and that at least one is vanishing. 
If $\delta_{\alpha_1}$ and $\delta_{\alpha_2}$ are both even, then neither $\vartheta_{\delta_{\alpha_1}}(\Omega_0,z')$ nor $\vartheta_{\delta{\alpha_2}}(\tau_0,z'')$ vanish at $0$. 
It follows that $\delta_{\alpha_1}$ and $\delta_{\alpha_2}$ must both be odd. 
Since odd thetanulls vanish identically, it follows that $\vartheta_{\delta_{\alpha_1}}(\Omega_0,z')$ and $\vartheta_{\delta{\alpha_2}}(\tau_0,z'')$ both vanish at $0$. 
Since the theta divisor on the jacobian of a curve of genus 1 or 2 is smooth, it follows that both of these functions vanish to first order at $0$. Thus $\text{ord}_0\Theta_{\alpha} = 2$, and by the heat equation $D_{\alpha}$ is smooth at $\Omega$.

$\hspace{.25in}$The second case to consider is that in which $\Omega\in \h_1^3$. As above, we are free to assume that $\Omega = \tau_1\oplus\tau_2\oplus \tau_3$ and that $\vartheta_{\delta_{\alpha}}(\Omega,0) = 0$. Under this assumption, we have
\begin{equation*}
\vartheta_{\delta{\alpha}}(\Omega,z) = \vartheta_{\delta_{\alpha_1}}(\tau_1,z_1)\vartheta_{\delta_{\alpha_2}}(\tau_2,z_2)\vartheta_{\delta_{\alpha_3}}(\tau_3,z_3).
\end{equation*}
Since $\delta_{\alpha}$ is even, then either all three of the $\delta_{\alpha_j}$ are even or two are odd and one is even. In the first case, none of the factors vanishes at 0. So it must be that two of the $\alpha_j$ are odd and one is even. Thus it follows that $\vartheta_{\delta_{\alpha}}(\Omega,z)$ vanishes to second order at $z = 0$. Again by the heat equation, $D_{\alpha}$ is smooth at $\Omega$.
\end{proof}
We are now in a position to describe the singular locus  $\h_3^{hyp,sing}$ of $\h_3^{hyp}$. 
\begin{corollary}[Hain, \cite{hain2015torelli}]
The singular locus $\h_3^{hyp,sing}$ of $\h_3^{hyp}$ is equal to $\h_3^{red}$.
\end{corollary}
\begin{proof}
Because the components of $\h_3^{hyp}$ are all smooth, the singular locus of $\h_3^{hyp}$ consists of those points $p$ which lie in some intersection $D_{\alpha}\cap D_{\beta}$ where $\alpha\neq \beta$. 

Since a smooth hyperelliptic curve $C$ has a unique vanishing even theta characteristic, any period matrix for $C$ lies in exactly one of the $D_{\alpha}$. Thus, any point $p\in D_{\alpha}\cap D_{\beta}$ must be the period matrix of a singular curve and $p\in \h_3^{red}$. 

On the other hand, if $p\in \h_3^{red}$ then $p$ is contained in the intersection of two distinct components of $\h_3^{hyp}$. In order to demonstrate this, we may first assume that $p\in \h_2\times \h_1$. Then $p = \Omega_0\oplus \tau_0$ for some $\Omega_0\in \h_2$ and $\tau_0\in \h_1$. Then for any choice of odd characteristic $\delta_1$ of genus 2, the characteristic $\delta = \delta_1\oplus (1/2,1/2)$ is even and 
\begin{equation*}
\vartheta_{\delta}(\Omega_0\oplus \tau_0,0) = \vartheta_{\delta_1}(\Omega_0,0)\vartheta_{(1/2,1/2)}(\tau_0,0) = 0\cdot 0 = 0.
\end{equation*}
This finishes the proof.
\end{proof}
\indent  Define $D_{\alpha}^{red}: = D_{\alpha}\cap \h_3^{red}$. It is a closed analytic subset of $\h_3$. We will need to make use of the following lemma in later sections. 
\begin{lemma}
The locus $D_{\alpha}^{red}\subset D_{\alpha}$ is a union of irreducible components of $\h_3^{red}$. 
\end{lemma}
\begin{proof}
Each component of $D_{\alpha}^{red}$ is a component of the zero locus of the restriction of a thetanull $\vartheta_{\delta_{\beta}}(-,0)$ to $D_{\alpha}$, where $\alpha\neq \beta$. Thus each component of $D_{\alpha}^{red}$ is a closed analytic subset of $\h_3$ of pure dimension 4. Let $X$ denote any one of these components. Then $X$ is contained in $\h_3^{red}$, the irreducible components of which are also 4-dimensional. It follows that $X$ is a component of $\h_3^{red}$.
\end{proof}

\subsection{\texorpdfstring{The Combinatorics of $\h_3^{hyp}$}{The Combinatorics of h3hyp}}
We will need to collect several facts which describe the combinatorics of the ways in which the components of $\h_3^{hyp}$ intersect eachother. 

\begin{lemma}[Hain, \cite{hain2015torelli}]
\begin{enumerate}
\item[]
\item If $p\in \h_3^{red}-\h_3^{red,sing}$, then exactly six components of $\h_3^{hyp}$ intersect at $p$. These six components meet pairwise transversely.
\item If $p\in \h_3^{red,sing}$, then exactly nine components of $\h_3^{hyp}$ intersect at $p$.
\end{enumerate}
\end{lemma}
\begin{proof}
\begin{enumerate}
\item[]
\item \hspace{.1in} We may assume that $p = \Omega_0\oplus \tau_0\in (\h_2-\h_2^{red})\times \h_1$. Let $\delta_{\alpha}$ be any even characteristic. First observe that for all $\Omega_1\oplus\Omega_2\in \h_2\times \h_1$ we have
\begin{equation*}
\vartheta_{\delta_{\alpha}}(\Omega_1\oplus\Omega_2,0) = \vartheta_{\delta_{\alpha_1}}(\Omega_1,0)\vartheta_{\delta_{\alpha_2}}(\Omega_2,0)
\end{equation*}
where $\delta_{\alpha_1},\delta_{\alpha_2}$ are characteristics such that $\delta_{\alpha} = \delta_{\alpha_1}\oplus\delta_{\alpha_2}$. Since $\delta_{\alpha}$ is even, it must be the case that either both of the $\delta_{\alpha_j}$ are odd, or that both are even. 

\hspace{.1in} Now suppose that $\vartheta_{\delta_{\alpha}}(\Omega_0\oplus\tau_0,0) = 0$. It must be that both $\delta_{\alpha_1}$ and $\delta_{\alpha_2}$ are odd, since if  both were even then both $\vartheta_{\delta_{\alpha_1}}(\Omega_0,0)$ and $\vartheta_{\delta_{\alpha_2}}(\tau_0,0)$ would be non-zero.
On the other hand, there is exactly one odd characteristic $(1/2,1/2)$ in genus 1 and there are exactly six odd characteristics $\delta_{\alpha_1}$ in genus 2. For any choice of odd characteristic $\delta_{\alpha_1}$, the thetanull $\vartheta_{\delta_{\alpha}}(-,0)$ vanishes at $p$, where $\delta_{\alpha} = \delta_{\alpha_1}\oplus (1/2,1/2)$. It follows that there are precisely six even thetanulls which vanish at $p$. In other words, exactly six components of $\h_3^{hyp}$ intersect at $p$.
\item 
\hspace{.1in} We may assume that $p = \tau_1\oplus\tau_2\oplus \tau_3 \in \h_1^3$. Let $\delta_{\alpha}$ be any even characteristic. For all $\Omega_1\oplus\Omega_2\oplus \Omega_3\in \h_1^3$ we have
\begin{equation*}
\vartheta_{\delta_{\alpha}}(\Omega_1\oplus\Omega_2\oplus\Omega_3,0) = \vartheta_{\delta_{\alpha_1}}(\Omega_1,0)\vartheta_{\delta_{\alpha_2}}(\Omega_2,0)\vartheta_{\delta_{\alpha_3}}(\Omega_3,0)
\end{equation*}
where the $\delta_{\alpha_j}$ are all genus 1 characteristics such that $\delta_{\alpha} = \delta_{\alpha_1}\oplus\delta_{\alpha_2}\oplus\delta_{\alpha_3}$. Since $\delta_{\alpha}$ is even, it must be the case that either all of the $\delta_{\alpha_j}$ are even or that two of the $\delta_{\alpha_j}$ are odd and that the remaining one is even.

\hspace{.1in} Now suppose that $\vartheta_{\delta_{\alpha}}(\Omega_0\oplus\tau_0,0) = 0$. Then it must be that two of the $\delta_{\alpha_j}$ are odd and that the remaining one is even. To see why, notice that if all of the $\delta_{\alpha_j}$ are even, then all of the numbers $\vartheta_{\delta_{\alpha_j}}(\tau_j,0)$ are non-zero. On the other hand, when the genus is 1, there are three even characteristics and one odd characteristic. Given a fixed choice of two odd characteristics $\delta_{\alpha_{j_1}},\delta_{\alpha_{j_2}}$ there are then three choices of even characteristic $\delta_{\alpha_{j_3}}$ which produce an even characteristic $\delta_{\alpha} = \delta_{\alpha_1}\oplus\delta_{\alpha_2}\oplus\delta_{\alpha_3}$.  For any such choices the resulting thetanull $\vartheta_{\delta_{\alpha}}(-,0)$ vanishes at $p$. It follows that exactly nine components of $\h_3^{hyp}$ intersect at $p$.

\end{enumerate}
\end{proof}
Note that if a thetanull $\vartheta_{\delta_{\alpha}}(-,0)$ vanishes at a point $p\in \h_3^{red}-\h_3^{red,sing}$, then it vanishes identically along the component of $\h_3^{red}$ which contains $p$. Via the $\Sp_3(\ZZ)$ symmetry of the filtration on $\h_3$, this is a consequence of the fact that if a thetanull $\vartheta_{\delta}(-,0)$ vanishes at a point $p\in (\h_2-\h_2^{red})\times \h_1$, then $\vartheta_{\delta}(-,0)$ vanishes identically on $\h_2\times \h_1$. This statement is readily verified by explicit computation with thetanulls, much as is done above.
\begin{lemma}
Suppose that $p\in D_{\alpha}\cap \h_3^{red,sing}$. Then $D_{\alpha}$ contains at least two of the three components of $\h_3^{red}$ which intersect at $p$.
\end{lemma}
\begin{proof}
We may assume without loss that $p\in \h_1^3$ and that $\alpha$ is a sum of genus 1 characteristics of the form $\alpha_1\oplus\alpha_2\oplus \alpha_3$ with $\alpha_1,\alpha_3$ odd characteristics and $\alpha_2$ an even characteristic. Then it is enough to show that $D_{\alpha}$ contains both of the components $\h_2\times\h_1$ and $\h_1\times \h_2$ of $\h_3^{red}$. But this is clearly true because for any elements $\Omega_0\in \h_2$ and $\tau_0\in \h_1$, we have
\begin{align*}
\vartheta_{\delta{\alpha}}(\Omega_0\oplus \tau_0, 0) = \vartheta_{\delta_{\alpha_1}\oplus\delta_{\alpha_2}}(\Omega_0,0)\vartheta_{\delta_{\alpha_3}}(\tau_0,0) = 0\\
\vartheta_{\delta_{\alpha}}(\tau_0\oplus\Omega_0, 0) = \vartheta_{\delta_{\alpha_1}}(\tau_0,0)\vartheta_{\delta_{\alpha_2}\oplus\delta_{\alpha_3}}(\Omega_0,0) = 0
\end{align*}
since $\alpha_1$ and $\alpha_3$ are odd.
\end{proof}
\begin{corollary}[Hain, \cite{hain2015torelli}]\label{corollary51}
Let $p\in D_{\alpha}\cap \h_3^{red,sing}$. Then $D_{\alpha}$ contains exactly two of the three components of $\h_3^{red}$ which intersect at $p$ inside of $\h_3$.
\end{corollary}
\begin{proof}
Three components of $\h_3^{red}$ intersect at any point of $\h_3^{red,sing}$, and 6 components of $\h_3^{hyp}$ intersect at any point of $\h_3^{red}-\h_3^{red,sing}$. This implies that 18 components of $\h_3^{hyp}$ intersect at any point of $\h_3^{red,sing}$. However, as we have shown, only 9 distinct components of $\h_3^{hyp}$ intersect at any point of $\h_3^{red,sing}$. This means that some of the components of $\h_3^{hyp}$ containing $p$ contain more than one of the components of $\h_3^{red}$ that contain $p$. 

Now let $p\in \h_3^{red,sing}$. Then any component of $\h_3^{hyp}$ containing $p$ also contains at least two of the three components of $\h_3^{red}$ which contain $p$, by the preceding lemma.  However, none of the components of $\h_3^{hyp}$ can contain three of the components of $\h_3^{red}$ containing $p$ because, in that case, there would be more than 18 components of $\h_3^{hyp}$ meeting at $p$ (counting multiplicity). This all implies that $D_{\alpha}$ contains exactly two of the components of $\h_3^{red}$ which intersect at $p$.
\end{proof}
\begin{proposition}
The divisor $D_{\alpha}^{red}\subset D_{\alpha}$ has simple normal crossings.
\end{proposition}
\begin{proof}
Since the components of $D_{\alpha}^{red}$ are components of $\h_3^{red}$, each component of $D_{\alpha}^{red}$ is non-singular. Now we need only show that $D_{\alpha}^{red}$ has normal crossings. 

Suppose that $p\in D_{\alpha}^{red}$ is a singular point. Then $p\in \h_3^{red,sing}$ and by the preceding proposition exactly two of the three components of $\h_3^{red}$ which intersect at $p$ are contained in $D_{\alpha}$. Since no two of these components are tangent in $\h_3$, this implies that the two components which lie in $D_{\alpha}$ intersect transversely. 
This completes the proof.
\end{proof}
\section{\texorpdfstring{The Integral Homology of the Components of $\h_3^{hyp}$}{The Homology of the components of h3hyp}}
\begin{remark}
In this section, all homology groups are taken with $\ZZ$-coefficients.
\end{remark}
In this section we carry out a computation of the groups $H_k(D_{\alpha})$ for $k\geq 3$ of an arbitrary component $D_{\alpha}$ of $\h_3^{hyp}$. Since there is a biholomorphism $D_{\alpha}\cong \cHH_3^c[0]$, this computation leads directly to the proof of Theorem \ref{1.3}. Unfortunately, a successful computation of $H_2(D_{\alpha})$ has remained elusive. 

The primary tool in this calculation is the PL Gysin sequence from \cite{hain2006finiteness}. Specifically, we will apply the long exact sequence
\begin{equation*}
\cdots \rightarrow H_c^{10-k-1}(D_{\alpha}^{red})\rightarrow H_k(D_{\alpha})\rightarrow H_k(D_{\alpha}-D_{\alpha}^{red})\rightarrow \cdots 
\end{equation*}
to the computation of the homology groups of $D_{\alpha}$. In order to carry this out, we will need to pin down the compactly supported cohomology of $D_{\alpha}^{red}$ and the homology of $D_{\alpha}-D_{\alpha}^{red}$. Fortunately, using the spectral sequence for compactly supported cohomology appearing in \cite{hain2006finiteness}, we will be able to compute the groups $H_c^k(D_{\alpha}^{red})$ exactly. Since $D_{\alpha}-D_{\alpha}^{red}$ is a $K(T\Delta_3,1)$ (c.f. \cite{brendle2014generators}, \cite{hain2006finiteness}), much of its homology can be computed using \citep[Main Theorem 1]{brendle2013cohomology}. Their result implies that the homology groups $H_k(T\Delta_3)$ all vanish when $k\geq 3$ and that $H_2(T\Delta_3)$ is a free abelian group.

\vspace{.1in}
Suppose that $X$ is a PL manifold, that $Y$ is a closed PL subspace that is a locally finite union 
$Y = \bigcup_{i\in I}Y_i$
of closed PL subspaces $X$, and that $I$ is a partially ordered set. Define
\begin{equation}
Y_{(i_0,\ldots, i_k)} = Y_{i_0}\cap\cdots \cap Y_{i_k}.
\end{equation}
and set 
\begin{equation}
\cY_k = \coprod_{i_0<\cdots <i_k}Y_{(i_0,\ldots, i_k)}.
\end{equation}
Then it is shown in \citep[Proposition 7]{hain2006finiteness} that there is a spectral sequence 
\begin{equation}\label{hainspectral}
E_1^{s,t} = H_c^t(\cY_s)\implies H_c^{s+t}(Y)
\end{equation}
where $\cY_s = \displaystyle \coprod_{\alpha_0<\cdots <\alpha_s} Y_{\alpha_0}\cap\cdots \cap Y_{\alpha_s}$.

\begin{lemma}\label{vanishinglemma}
The groups $H_c^k(D_{\alpha}^{red})$ vanish unless $k=7$ or $k=8$, in which case they are free abelian.
\end{lemma}
\begin{proof}
We apply the spectral sequence (\ref{hainspectral}) with $X= D_{\alpha}$ and $Y = D_{\alpha}^{red}$ and an arbitrary ordering on the set $I$ of components of $Y$. Since $D_{\alpha}^{red}$ is a simple normal crossings divisor in $D_{\alpha}$, the subsets $\cY_s$ are all empty when $s\geq 2$ by Corollary \ref{corollary51}. Consequently, the $E_1$ page of the spectral sequence
\begin{equation}
E_1^{s,t} = H_c^t(\cY_s)\implies H_c^{s+t}(D_{\alpha}^{red})
\end{equation}
is very sparse. Since the components of $D_{\alpha}^{red}$ and $D_{\alpha}^{red,sing}$ are diffeomorphic to $\h_2\times \h_1\cong \RR^8$ and $\h_1^3\cong \RR^6$, respectively, the only non-trivial portion of the $E_1$ page is
\begin{equation*}
\begin{tabular}{l|r}
 & \xymatrix{
  H^8_c(\cY_0)&0  \\
 0 &0 \\
  0&H_c^6(\cY_1) \\
  \vdots& \vdots    \\
  0 & 0
  }\\
\hline\\
\end{tabular}
\end{equation*}
Because of the shape of the $E_1$ page, the differentials in this spectral sequence all vanish. The claim now follows immediately. 
\end{proof}
\begin{proposition}\label{endprop}
The groups $H_k(D_{\alpha})$ vanish for all $k\geq 4$ and $H_3(D_{\alpha})$ is free abelian.
\end{proposition}
\begin{proof}
For each $k\geq 0$ the PL Gysin sequence has segments
\begin{equation*}
H_c^{10-k-1}(D_{\alpha}^{red})\rightarrow H_k(D_{\alpha}-D_{\alpha}^{red})\rightarrow H_k(D_{\alpha})\rightarrow H_c^{10-k}(D_{\alpha}^{red}).
\end{equation*}
Combined with the fact that $H_k(D_{\alpha}) = 0$ for $k\geq 3$ and the result of Lemma \ref{vanishinglemma}, we immediately find that $H_k(D_{\alpha}) = 0$ as long as $k\geq 4$. When $k = 3$, we get an injection
\begin{equation*}
0\rightarrow H_3(D_{\alpha})\rightarrow H_c^7(D_{\alpha}^{red}).
\end{equation*}
Since $H_c^7(D_{\alpha}^{red})$ is free abelian by the spectral sequence computation, $H_3(D_{\alpha})$ is also free abelian.
\end{proof}
\begin{proof}[Proof of Theorem 1.3]
The claim follows at once from Proposition \ref{endprop} and the fact that $D_{\alpha}\cong \cHH_3^c[0]$ is simply-connected. 
\end{proof}

\vspace{.1in}
At this point, we would like to address Question 2 and briefly discuss a necessary and sufficient condition for $\cHH^c_3[0]$ to be homotopy equivalent to a bouquet of spheres. In view of the isomorphisms $H_k(D_{\alpha}-D_{\alpha}^{red})\cong H_k(\cHH_3[0]) \cong H_k(T\Delta_3)$ and the fact that $D_{\alpha}$ is simply-connected, the PL Gysin sequence has a segment
\begin{align*}
0\rightarrow H_3(D_{\alpha})\rightarrow  H_c^7(D^{red}_{\alpha}) \rightarrow &H_2(T\Delta_3)\\
& \rightarrow H_2(D_{\alpha})\rightarrow H_c^8(D_{\alpha}^{red})\rightarrow H_1(T\Delta_3)\rightarrow 0
\end{align*}
Since $H_2(T\Delta_3)$ is free abelian, if it were true that the image of $H_c^7(D_{\alpha})$ in $H_2(T\Delta_3)$ were a direct summand, then it would follow that $H_2(D_{\alpha})$ is also free abelian. Since $H_3(D_{\alpha})$ is free abelian, it follows from the Hurewicz Theorem that, if this did occur, there would be a continuous map
$$
\bigvee_{I_2} S^2\vee \bigvee_{I_3}S^3 \rightarrow D_{\alpha}\ \ \ \ \ \ I_2,I_3\ \text{some index sets with $I_2$ infinite}
$$
inducing isomorphisms on integral homology groups. By the homological form of Whitehead's Theorem, such a map would necessarily be a homotopy equivalence. On the other hand, if such a homotopy equivalence were to exist, the image of the boundary map $H_c^7(D_{\alpha}^{red})\rightarrow H_2(T\Delta_3)$ would have to be a direct summand, since then its cokernel would have to be a subgroup of the free abelian group $H_2(D_{\alpha})$. We summarize this discussion as follows.
\begin{proposition}
The component $D_{\alpha}$ is homotopy equivalent to a bouquet of 2-spheres and 3-spheres if and only if the image of $H_c^7(D_{\alpha})$ in $H_2(T\Delta_3)$ is a direct summand.
\end{proposition}
%
%
%
%
%

\vspace{.15in}
\noindent Kevin Kordek\\
Department of Mathematics\\
Mailstop 3368\\
Texas A$\&$M University\\
College Station, TX 77843-3368\\
 E-mail: \sf{kordek@math.tamu.edu}
\end{document}